\documentclass{amsart}
\numberwithin{equation}{section}

\newtheorem{thm}{Theorem}[section]
\newtheorem{lem}[thm]{Lemma}
\newtheorem{cor}[thm]{Corollary}

\begin{document}
\title{Proceedings of the American Mathematical Society}

\author{Ravshan Ashurov}
\address{National University of Uzbekistan named after Mirzo Ulugbek and Institute of Mathematics, Uzbekistan Academy of Science}
\curraddr{Institute of Mathematics, Uzbekistan Academy of Science,
Tashkent, 81 Mirzo Ulugbek str. 100170} \email{ashurovr@gmail.com}

\small

\title[Generalized localization]
{Generalized localization for spherical partial sums of multiple
Fourier series}

\begin{abstract}

In this paper the generalized localization principle for the
spherical partial sums of the multiple Fourier series in the $L_2$ -
class is proved, that is, if $f\in L_2(T^N)$ and $f=0$ on an open
set $\Omega \subset T^N$, then it is shown that the spherical
partial sums of this function converge to zero almost - everywhere
on $\Omega$. It has been previously known that the generalized
localization is not valid in $L_p(T^N)$ when $1\leq p<2$.
Thus the problem of generalized localization for the spherical
partial sums is completely solved in $L_p(T^N)$, $p\geq
1$: if $p\geq2$ then we have the generalized localization and if
$p<2$, then the generalized localization fails.

\end{abstract}

\maketitle

\section{Introduction }

Let $\{f_n\}$, $n\in Z^N$, be the Fourier coefficients of a
function $f\in L_2(T^N)$, $N\geq2$. Consider the spherical partial
sums of the multiple Fourier series:
\begin{equation}\label{SL}
S_\lambda f(x)=\sum\limits_{|n|^2< \lambda}f_n\,e^{inx}.
\end{equation}

The aim of this paper is to investigate convergence
almost-everywhere (a.e.) of these partial sums. One of the first
questions which arise in the study of a.e. convergence of the sums
(\ref{SL}) is the question of the validity of the Luzin
conjecture: is it true that the spherical sums (\ref{SL}) of the
Fourier series of an arbitrary function $f\in L_2(T^N)$ converge
a.e. on $T^N$? In other words, does Carleson's theorem extend to
$N$-fold Fourier series when the latter is summed spherically? The
answer to this question is open so far. What is known is only that
Hunt's theorem (convergence a.e. for $L_p$ functions) does
not extend to $N$-fold ($N \geq 2$) series summed by circles (see
\cite{AAP} and references therein). Historically progress with
solving the Luzin conjecture has been made by considering easier
problems. One of such easier problems is to investigate
convergence a.e. of the spherical sums (\ref{SL}) on $T^N\setminus
supp f$.

Il'in \cite{IL} was the first to introduce the concept of
generalized principle of localization for an arbitrary
eigenfunction expansions. Following Il'in we say that the
generalized localization principle for $S_\lambda$ holds in
$L_p(T^N)$, if for any function $f\in L_p(T^N)$ the equality
\begin{equation}\label{PV}
\lim\limits_{\lambda\to\infty}S_\lambda f(x)=0
\end{equation}
holds a.e. on $T^N\setminus supp f$.

Observe, unlike the classical Riemann localization principle, here
it suffices the equality (\ref{PV}) to be hold only a.e. (not
everywhere) on $T^N\setminus supp f$.

For the spherical partial integrals of multiple Fourier integrals
(we denote by $\sigma_\lambda f(x)$) the generalized localization
principle in $L_p(R^N)$ has been investigated by many authors
(see \cite{BAS2}-\cite{SJ}). In particular,  in the remarkable
paper of A. Carbery and F. Soria  \cite{CAR} the validity of the
generalized localization for $\sigma_\lambda$ has been proved in
$L_p(R^N)$ when $2\leq p<2N/(N-1)$. Note, that the method
introduced by these authors can be easily applied to non-spherical
partial integrals too \cite{AAA}.

If we turn back to the multiple Fourier series (\ref {SL}) and
consider the classes $L_p(T^N)$ when $1\leq p<2$, then as A.
Bastys \cite{BAS}  has proved, following Fefferman in making use
of the Kakeya's problem, that the generalized localization for
$S_\lambda$ is not valid, i.e. there exists a function $f\in
L_p(T^N)$, such that on some set of positive measure, contained in
$T^N\backslash supp f,$ we have
$$
\overline{\lim\limits_{\lambda\to\infty}}|S_\lambda f(x)|=+\infty.
$$
It may be worth mentioning that in \cite{BAS} this
result is also proved for the spherical partial integrals
$\sigma_\lambda f(x)$.

The main result of this paper is the following statement.

\begin{thm}\label{MAIN}
Let $f\in L_2(T^N)$ and $f=0$ on an open set $\Omega\subset T^N$.
Then the equality (\ref{PV}) holds a.e. on $\Omega$.
\end{thm}

Thus the problem of generalized localization for $S_\lambda$ is
completely solved in classes $L_p(T^N)$, $p\geq 1$: if $p\geq2$
then we have the generalized localization and if $p<2$, then the
generalized localization fails.

In the study of a.e. convergence it is convenient to introduce the
maximal operator
$$ S_\star f(x)=\sup\limits_{\lambda>0}|S_\lambda
f(x)|. $$

The prove of Theorem \ref{MAIN} is based on the following estimate
of this operator.

\begin{thm}\label{MAX}
Let $f\in L_2(T^N)$ and $f=0$ on the ball $\{|x|<R\}\subset T^N$.
Then for any $r<R$ there exists a constant $C=C(R,r)$, such that
\begin{equation}\label{MES}
\int\limits_{|x|\leq r}|S_\star f(x)|^2 dx\leq C
\int\limits_{T^N}|f(x)|^2 dx.
\end{equation}
\end{thm}

The formulated  theorems are easily transferred to the case of non-spherical partial sums of multiple Fourier series (see \cite{AAA}, \cite{AB}).

\section{Auxiliary assertions}

The proofs of Theorems \ref{MAIN} and \ref{MAX} are based on several auxiliary assertions,
which are given in this section. Here we have borrowed some original
ideas from A. Carbery and F. Soria \cite{CAR}, where the authors
have investigated the multiple Fourier
integrals.

So we assume that $f=0$ on the fixed ball $\{|x|<R\}\subset T^N$
and fix a number $r<R$.

Let $\chi_b(t)$ be the characteristic function of the segment
$[0,b]$. We denote by $\varphi_1(t)$ a smooth function with
$\chi_{(R-r)/3}(t)\leq \varphi_1(t)\leq \chi_{2(R-r)/3}(t)$ and
put $\varphi_2(t)=1-\varphi_1(t)$. Now we define a new function
$\psi(x)$ as follows: $\psi(x)=\varphi_2(|x|)$, when $x\in T^N$
and otherwise it is a $2\pi$ - periodical on each variable $x_j$
function.

Let us denote
$$
\theta(x,\lambda)=(2\pi)^{-N}\sum\limits_{|n|^2< \lambda} e^{inx}.
$$
Then by definition of the Fourier coefficients we may write
$$
S_\lambda f(x)=\int\limits_{T^N}\theta(x-y,\lambda)f(y)dy.
$$
If we define $\theta_\lambda(x)=\theta(x,\lambda)\psi(x)$, then we
have
$$
S_\lambda f(x)=\int\limits_{T^N}\theta_\lambda(x-y)f(y)dy, \,\,for
\,\,all\,\, x,\,\,with \,\,|x|\leq r,
$$
since $f$ is supported in $\{|x|\geq R\}$. Therefore to prove the
estimate (\ref{MES}) it suffices to obtain the inequality
\begin{equation}\label{theta}
\int\limits_{T^N}\sup_{q >\,0}\left|\theta_q\ast f\right|^2dx\leq
C\int\limits_{T^N}|f(x)|^2dx,
\end{equation}
where $\sup$ is taken over all integers.

Now we need some estimates for the Fourier coefficients of the
function $\theta_k(x)$, which we denote by $(\theta_k)_n$.

\begin{lem}\label{coef2}
For an arbitrary integer $l$ there exists a constant $C_l$,
depending on $l, \, r $ and $R$, such that for all $k\geq 0$ and
$n\in Z^N$ one has
$$
|(\theta_k)_n|\leq\frac{C_l}{(1+||n|-\sqrt{k}|)^l}.
$$
\end{lem}

\begin{proof} Let $\{\psi_m\}$ be the Fourier coefficients of the function
$\psi(x)$. Then
$$
(\theta_k)_n=(2\pi)^{-2N}\int\limits_{T^N}\sum\limits_{|m|<\sqrt{k}}
e^{imx}\psi(x)e^{-inx}dx=(2\pi)^{-N}\sum\limits_{|n-m|<\sqrt{k}}\psi_{m}.
$$
If $|n|>\sqrt{k}$ then we have
$$
(2\pi)^{-N}|\sum\limits_{|n-m|<\sqrt{k}}\psi_{m}|\leq
(2\pi)^{-N}\sum\limits_{|m|>|n|-\sqrt{k}}|\psi_{m}|.
$$
Similarly, if $|n|\leq \sqrt{k}$ then making use of the equality
(observe, $\psi$ is an infinitely differentiable and $2\pi$ -
periodical function) $\sum\psi_m=\psi(0)=0$, we obtain
$$
(2\pi)^{-N}|\sum\limits_{|n-m|<\sqrt{k}}\psi_{m}|=
(2\pi)^{-N}|-\sum\limits_{|n-m|\geq\sqrt{k}}\psi_{m}|\leq
(2\pi)^{-N}\sum\limits_{|m|\geq\sqrt{k}-|n|}|\psi_{m}|.
$$

Now it is sufficient to  note that for any integer $j\geq 0$ there
exists a constant $c_j$, depending on $(R-r)$, such that
\begin{equation}\label{psi}
|\psi_m|\leq\frac{c_j}{(1+|m|)^j},
\end{equation}
and to estimate the last sum by comparing it with the
corresponding integral.
\end{proof}

We will apply the estimate (\ref{psi}) further, so the
corresponding constants will depend on $r$ and $R$. In addition,
as we have done above, in order to estimate number series we
compare them with the corresponding integrals.

Let $(\Theta_j)_n=(\theta_{j+1})_n-(\theta_{j})_n$, that is,
$$
(\Theta_{j})_n= (2\pi)^{-N}\sum\limits_{
|m|^2=j}\psi_{m-n}=(2\pi)^{-N}\sum\limits_{
|n-m|^2=j}\psi_{m}
$$
(if the Diophantine equation $|m-n|^2=j$ does not have a solution,
then $(\Theta_{j})_n=0$). These numbers have a better estimate
than $(\theta_j)_n$ in the following sense. Suppose
$k\leq\sqrt{j}< k+1$, i.e. $k^2\leq j< k^2+2k+1$, or
$j=k^2+p,\,0\leq p< 2k+1$, then according to Lemma \ref{coef2},
$(\theta_j)_n$ has the same estimate. But, as we will see below,
the numbers $(\Theta_j)_n$ vanish in the same interval  in some
sense. In particular, the following statement is true.

\begin{lem}\label{bigl}

For any $l$, there exists a constant $C_l$ such that
\begin{equation}\label{Big}
\sum\limits_{k\leq\sqrt{j}<k+1}|(\Theta_j)_n|^2\leq
\frac{C_l}{(1+||n|-k|)^l}.
\end{equation}
\end{lem}

\begin{proof} Let $|n|\leq k$; otherwise estimates are similar. By virtue of estimate (\ref{psi}) we have
$$
\sum\limits_{k\leq\sqrt{j}<k+1}|(\Theta_j)_n|\leq
(2\pi)^{-N}\sum\limits_{ k\leq |n-m|<k+1}|\psi_{m}|\leq
$$
$$
\leq
(2\pi)^{-N}\sum\limits_{|m|>||n|-k|}\,\,\frac{c_j}{(1+|m|)^j}\leq
\frac{C_l}{(1+||n|-k|)^l}.
$$
Since $|(\Theta_{j})_n|^2\leq C|(\Theta_{j})_n|$, Lemma is proved.
\end{proof}

\begin{cor}\label{LBig} Uniformly on $n$ one has

$$\sum\limits_{j=0}^\infty|(\Theta_{j})_n|^2=\sum\limits_{k=0}^\infty\sum\limits_{k\leq\sqrt{j}<k+1}|(\Theta_j)_n|^2\leq C.$$

\end{cor}

If we properly group the numbers $(\Theta_{k^2+p})_n$  by parameter $p$, then
a stronger result than Lemma \ref{bigl} can be obtained. Our
nearest aim is to implement this grouping.

Denote by $y_0$ (the nearest one to the origin) the intersection point of
the ball $\{x\in R^N: |x-n|\leq k+1\}$ with the straight line $On$ that passes through the origin and point $n$. Let $T_{y_0}$ be the
tangential hyperplane to the ball $\{x\in R^N: |x-n|\leq k+1\}$ at
the point $y_0$. Let $B_0:=\{y\in T_{y_0}: |y-y_o| < 1\}$ and $B_j:=\{y\in T_{y_0}:
\sqrt{j}\leq |y-y_o| < \sqrt{j+1}\,\}$, where $j=1,2, \cdot\cdot\cdot, 2k-1$. Let $C^k_j$, $j=0, 1,
\cdot\cdot\cdot, 2k-1$, be the $N-$ dimensional cylinders with the base
$B_j$ and with the axis parallel to $On$ and the length $|n|$.
Consider the ring $K=\{x\in R^N: k\leq |x-n| < (k+1)\}$
and divide it in to the following sets: $P_j^k = K\cap C_j^k$,
$j=0, 1, \cdot\cdot\cdot, 2k-1$.

Let us define the sets $Q_q^k$, $q =0, 1, \cdot\cdot\cdot, 2k-1$,
as follows. Let $Q_q^k$ be the set of those integers $p$, $0\leq
p\leq 2k$, for which the Diophantine equation $|m-n|^2=k^2+p$ has a solution in $P_q^k$. If $P_q^k$
does not contain any of solutions of equation $|m-n|^2=k^2+p$, for any $p$,
then we assign to the set $Q_q^k$ one of those parameters $p$ that are not included in the previous sets $Q_j^k$,
$j =0, 1, \cdot\cdot\cdot, q-1$. If there are no such
 $p's$ left, then we define $Q_j^k$, $j =q, q+1, \cdot\cdot\cdot,
2k-1$ as empty set.

In the proof of Lemma \ref{Lsmall} we need to know  how many at most
parameters $p$ does the set $Q_q^k$ contain. Observe, if we fix
$y\in On$, then the Diophantine equation $m\in Z^N, |m-y|^2=t,\,\,
q\leq t< q+1$ may have a solution only for one $t$ (note, in fact,
here it suffices to consider the "projection" of this equation
onto the hyperplane passing through the point $y$ and parallel to
$T_{y_0}$). The length of the projection of $P_q^k$ on the axis of
$Ox_1$ is less than $2\sqrt{q+1}$; (without loss of generality, we
can assume that the angle between $On$ and $Ox_1$ is less than or
equal to $\frac{\pi}{4}$). Consequently, if, for a fixed $p$,
there is a solution of the Diophantine equation $|m-n|^2=k^2+p$,
provided $m\in P_q^k$, then the first coordinates $m_1$ of the
numbers $m$, take less than $[2\sqrt{q+1}\,]$ ($[a]$ is the
integer part of the number $a$) different values. When $p$ varies
from $0$ to $2k$, then each of these numbers $m_1$ can take at
most two adjacent integer numbers. Hence each set $Q_q^k$ has less
than $4\sqrt{q+1}$ parameters $p$ with the above property .

With this choice of $Q_q^k$ we have the following statement.

\begin{lem}\label{nk} Let $q =0,1,\cdot\cdot\cdot, 2k-1$ and $S_p = \{m\in Z^N:
|m-n|^2 = k^2+p\}$ ($p = 0, 1, \cdot\cdot\cdot, 2k$). If $|n|\geq
k+1$, then
\begin{equation}\label{mn2}
\min\limits_{m\in S_p,\,p\in Q_q^k}|m|\geq \sqrt{(|n|-k-1)^2+q}.
\end{equation}
If $k< |n|< k+1$, then
$$
\min\limits_{m\in S_p,\,p\in Q_q^k}|m|\geq \sqrt{q}.
$$
If $|n|\leq k$, then
$$
\min\limits_{m\in S_p,\,p\in Q_q^k}|m|\geq
\frac{1}{2}\sqrt{(|n|-k)^2+q}.
$$
\end{lem}

\begin{proof}  Note that it is sufficient to estimate the minimum distance from the origin to the set $P_q^k$. If $|n|\geq
k+1$, then it is not hard to verify that the distance from the origin to the set $B_q$ is equal to $\sqrt{(|n|-k-1)^2+q}$.
Obviously, this value is less or equal to the distance between the origin and $P_q^k$. In case of $k< |n|< k+1$ 
arguments are similar.

If $|n|\leq k$, then minimum distance from
the origin to the set $P_q^k$ is less than or equal to $\sqrt{(|n|-\sqrt{k^2-q})^2+q}$. But we can estimate 
this number from below by $\frac{1}{2}\sqrt{(|n|-k)^2+q}$.  \end{proof}

As we mentioned above for $(\Theta_{j})_n$ one has a more stronger
result than Lemma \ref{bigl}.

\begin{lem}\label{Sbigl}

For any $l$, there exists a constant $C_l$ such that
\begin{equation}\label{SBig}
\sum\limits_{q=0}^{2k-1}(q+1)^2\sum\limits_{p\in
Q_q^k}|(\Theta_{k^2+p})_n|^2\leq \frac{C_l}{(1+\sqrt{||n|-k|})^l}.
\end{equation}
\end{lem}

\begin{proof} From the definition of $(\Theta_{j})_n$ one has
$$
\sum\limits_{q=0}^{2k-1}(q+1)\sum\limits_{p\in
Q_q^k}|(\Theta_{k^2+p})_n|\leq
(2\pi)^{-N}\sum\limits_{q=0}^{2k-1}(q+1)^2\sum\limits_{p\in
Q_q^k}\,\,\,\sum\limits_{|m-n|^2=k^2+p}|\psi_m|\leq
$$
(and by virtue of estimates (\ref{psi}) and  (\ref{mn2}) (we
assume that $|n|\geq k+1 $; otherwise arguments are similar) we
finally have)

$$
\leq \sum\limits_{q=0}^{2k-1}(q+1)^2\sum\limits_{|m|\geq
\sqrt{(|n|-k-1)^2+q}}\,\,\frac{c_j}{(1+|m|)^j}\leq
\frac{C_l}{(1+\sqrt{||n|-k|})^l}.
$$

Now (\ref{SBig}) follows from the estimate $|(\Theta_{j})_n|^2\leq
C|(\Theta_{j})_n|$.
\end{proof}

Next statement is an easy consequence of this Lemma.

\begin{cor}\label{CSbigl} Uniformly on $n$, one has
\begin{equation}\label{CSBig}
\sum\limits_{k=0}^{\infty}\sum\limits_{q=0}^{2k-1}(q+1)^2\sum\limits_{p\in
Q_q^k}|(\Theta_{k^2+p})_n|^2\leq C.
\end{equation}

\end{cor}

Now we turn back to the Fourier coefficients $(\theta_{j})_n$.
From Lemma \ref{coef2} we have the following estimate.

\begin{lem}\label{Lsmall}

Uniformly on $n$, one has
\begin{equation}\label{small}
\sum\limits_{k=0}^{\infty}\sum\limits_{q=0}^{2k-1}(q+1)^{-2}\sum\limits_{p\in
Q_q^k}|(\theta_{k^2+p})_n|^2\leq C.
\end{equation}

\end{lem}

\begin{proof} As we mentioned above, each $Q_q^k$ has less than $4\sqrt{q+1}$ parameter $p$. Therefore, by virtue of Lemma \ref{coef2}
one has
$$
\sum\limits_{k=0}^{\infty}\sum\limits_{q=0}^{2k-1}(q+1)^{-2}\sum\limits_{p\in
Q_q^k}|(\theta_{k^2+p})_n|^2\leq
\sum\limits_{k=0}^{\infty}\frac{C_l}{(1+||n|-k|)^l}\sum\limits_{q=0}^{2k-1}\frac{4\sqrt{q+1}}{(q+1)^{2}}\leq C.
$$
\end{proof}

\section{Proofs of Theorems}

First,  we prove the estimate (\ref{theta}). Let
$\Theta_j(x)=\theta_{j+1}(x)-\theta_j(x)$.  Then $\theta_{j+1}\ast
f+\theta_j\ast f=2\, \theta_j\ast f+\Theta_j\ast f$. Note the
Fourier coefficients of the function $\Theta_j(x)$ are the numbers
$(\Theta_j)_n$, introduced above.

If for a sequence of numbers $\{F_q\}$ we have $F_0=0$, then
$$
F_q^2=\sum\limits_{j=0}^{q-1}[F_{j+1}-F_j][F_{j+1}+F_j],\,\, q\geq
1.
$$
Hence
$$
[\theta_q\ast f]^2 =\sum\limits_{j=0}^{q-1}[\Theta_j\ast
f]^2+2\,\sum\limits_{j=0}^{q-1}
[\Theta_j\ast f] [\theta_j\ast f],
$$
or
$$
\sup_{q >\,0}\left|\theta_q\ast
f\right|^2\leq\sum\limits_{j=0}^{\infty}|\Theta_j\ast
f|^2+2\,\sum\limits_{k=0}^{\infty}\sum\limits_{q=0}^{2k-1}\sum\limits_{p\in
Q_q^k}|\Theta_{k^2+p}\ast f|(q+1)|\theta_{k^2+p}\ast f|(q+1)^{-1}.
$$

Integrating over $T^N$ and making use of the inequality $2ab\leq
a^2+b^2$ one has

$$
\int\limits_{T^N}\sup_{q >\,0}\left|\theta_q\ast
f\right|^2\leq\sum\limits_{n}|f_n|^2\sum\limits_{j=0}^\infty|(\Theta_{j})_n|^2+
$$
$$
+\sum\limits_{n}|f_n|^2\sum\limits_{k=0}^{\infty}\sum\limits_{q=0}^{2k-1}(q+1)^2\sum\limits_{p\in
Q_q^k}|(\Theta_{k^2+p})_n|^2+
$$
$$
+\sum\limits_{n}|f_n|^2\sum\limits_{k=0}^{\infty}\sum\limits_{q=0}^{2k-1}(q+1)^{-2}\sum\limits_{p\in
Q_q^k}|(\theta_{k^2+p})_n|^2\leq
$$
(making use of Corollaries \ref{LBig} and \ref{CSbigl} and  Lemma
\ref{Lsmall} and since $f$ is $L_2$ - function)
$$
\leq C\sum\limits_n|f_n|^2=C\int\limits_{T^N}|f(x)|^2 dx.
$$

Thus, the estimate (\ref{theta}) and, consequently,  Theorem
\ref{MAX} is proved.

Now we prove Theorem \ref{MAIN}. So let $f\in L_2(T^N)$ and $f=0$
on an open set $\Omega\subset T^N$. We extend $f(x)$ to outside of
$T^N$ $2\pi$ - periodically on each variable $x_j$. In these
conditions we must prove that the equality (\ref{PV}) holds a.e.
on $\Omega$. If $x\in\Omega$ an arbitrary point, then to do this
it suffices to show validity of (\ref{PV}) a.e. on a ball with
center at $x$ and sufficiently small radius $R$, so that this ball
belongs to $\Omega$. Therefore without loss of generality we may
suppose, that $f$ is supported outside of this ball or by
translation invariance, $f$ is supported in $\{|x|\geq R\}$, and
prove convergence to zero of $S_\lambda f(x)$ a.e. on the ball
$\{|x|<r\}$ for any $r<R$. But this statement can be proved by a
standard technique based on Theorem \ref{MAX} (see \cite{ST}). Thus
Theorem \ref{MAIN} is also proved.

\section{ Acknowledgement} The author conveys thanks to Sh. A. Alimov
for discussions of this result and gratefully acknowledges Marcelo M. Disconzi (Vanderbilt
University, USA) for
support and hospitality.

 The author was supported by Foundation for Support of Basic Research of the Republic of Uzbekistan
 (project number is OT-F4-88).

\

\

\bibliographystyle{amsplain}

\end{document}